\documentclass[10pt,a4paper]{article}
\usepackage[utf8]{inputenc}
\usepackage{amsmath, amsthm, tikz, colonequals, multicol, hyperref}
\usepackage{amsfonts}
\usepackage{amssymb}

\theoremstyle{plain}
\newtheorem{theorem} {Theorem} 
\newtheorem{lemma} [theorem] {Lemma}

\newtheorem{definition} [theorem]{Definition}

\theoremstyle{definition}
\newtheorem{example} [theorem]{Example}
\newtheorem*{example*}{Example}
\theoremstyle{remark}

\newcommand\xqed[1]{%
  \leavevmode\unskip\penalty9999 \hbox{}\nobreak\hfill
  \quad\hbox{#1}}
\newcommand\demo{\xqed{$\dashv$}}

\title{Central binomial coefficients also count $(2431,4231,1432,4132)$-avoiders\thanks{The author was supported by the Austrian Science Foundation FWF, grant P25337-N23.}}
\author{Marie-Louise Bruner
\medskip\\
\texttt{marie-louise.bruner@tuwien.ac.at}\medskip\\
Vienna University of Technology
}
\date{\today}

\begin{document}
\maketitle

\begin{abstract} This short paper is concerned with the enumeration of permutations avoiding the following four patterns: $2431$, $4231$, $1432$ and $4132$.
Using a bijective construction, we prove that these permutations are counted by the central binomial coefficients.
\end{abstract}

The first systematic analysis of permutations avoiding some  patterns was performed by Simion and Schmidt in 1985~\cite{simion1985restricted}.
Since then, the study of permutation patterns has become a bustling field of combinatorics as witnessed, amongst others,  by the monographs by B\'{o}na~\cite{bona_combinatorics_2004} and Kitaev~\cite{kitaev2011patterns}.
One says that a permutation $\tau$ of length $n$ contains a permutation $\pi$ of length $k \leq n$ as a pattern if there is a subsequence of length $k$ of $\tau$ that is order-isomorphic to $\pi$.
For example, the permutation $\tau= 24135$ contains the patterns $\pi=123$ as can be seen by considering the entries $2, 4$ and $5$ of $\tau$. 
However, $\tau$ avoids the pattern $\sigma=321$ since it does not contain a decreasing subsequence of length three.

One goal within the topic of permutation patterns is to find exact enumeration formul{\ae} for permutation classes defined by a finite set of forbidden patterns.
For a finite set $\Pi$ of permutation patterns, we denote by $\mathcal{S}_n(\Pi)$ the set of permutations of length $n$ that avoid all patterns in $\Pi$ and are interested in the cardinality $S_n(\Pi)=|\mathcal{S}_n(\Pi)|$.
In this short paper we describe a permutation class defined by four forbidden patterns of length four for which a very simple enumeration formula can be given.
Indeed, the set of permutations of length $n$ avoiding the patterns $2431$, $4231$, $1432$ and $4132$ simultaneously is counted by the central binomial coefficients:

\begin{theorem}
For all $n \in \mathbb{N}$ the following holds:
\[
S_n(2431,4231,1432,4132)=\binom{2 \cdot (n-1)}{n-1}=b_{n-1}.
\]
\label{CentralBinomial/thm:cbc}
\end{theorem}

In the following, let us denote the set of permutations avoiding the forbidden patterns $2431$, $4231$, $1432$ and $4132$ by $\mathcal{S}$ and the subset of $\mathcal{S}$ consisting of permutations of length $n$ by $\mathcal{S}_n$.

We shall prove this theorem by identifying every permutation in $\mathcal{S}_n$ by a certain \emph{code word} of length $(n-1)$. Later on we count these code words and show that they are precisely enumerated by the central binomial coefficients.
This is done by identifying the non-initial segments of code words with certain lattice paths.

The set of code words corresponding to permutations in $\mathcal{S}$ is defined in the following way:

\begin{definition}
$\mathcal{W} \subseteq \left( \left\lbrace 2, 3, \ldots\right\rbrace \cup \left\lbrace B, E \right\rbrace \right)^\mathbb{N} $ 
is the set of words $W=w_1 w_2 \ldots$ fulfilling the following conditions:
\begin{enumerate}
\item $w_i \in \left\lbrace B, E \right\rbrace \cup \left\lbrace j \in \mathbb{N}: 2 \leq j \leq i \right\rbrace$,
\item if $w_i \neq B, E$ then $w_{i+1} \neq B, E$,
\item if $w_i \neq B, E$ then $w_{i+1} \geq w_i$.
\end{enumerate}
We shall denote by $\mathcal{W}_n$ the set of all code words of length $n$ and by $W_n$ its cardinality.
\label{CentralBinomial/def:codewords}
\end{definition}

In a code word, the letters $B$ and $E$ stand for \emph{beginning} and \emph{end}, respectively.
Integers in a code word represent positions in the corresponding permutation. 
The precise meaning of this will become clear in the proof of the following Lemma~\ref{CentralBinomial/lem:code_words_and_perms}.

Form the definition above it follows that code words contain an \textit{initial segment} of letters $B$ and $E$ (conditions 1 and 3).
This initial segment is followed by a non-decreasing sequence of not too large integers (conditions 1 and 2).

\begin{example}
The six code words of length two are: $BB$, $BE$, $EB$, $EE$, $B2$, $E2$.
For an example of a code word of length $8$, see Figure~\ref{CentralBinomial/fig:example_path} on page~\pageref{CentralBinomial/fig:example_path}.
\demo
\end{example}

\begin{lemma}
There is a bijection between code words in $\mathcal{W}_{n-1}$ and permutations in $\mathcal{S}_n$.
\label{CentralBinomial/lem:code_words_and_perms}
\end{lemma}

The approach used to prove the lemma above can be seen as a \textit{generating trees} approach. 
This is a technique introduced in the study of Baxter permutations and systematized by Julian West \cite{west1996generating} in the context of pattern avoidance. 
It is frequently used for the enumeration of permutation classes, see e.g.\ \cite{bousquet2003four}.
However, we decided not to introduce the generating trees terminology here:
The following arguments and especially the enumeration of code words can be done in a bijective manner and do not require the use of rewriting rules and of the associated generating functions.

\begin{proof}%
Given a permutation $\pi$ in  $\mathcal{S}_n$, we can associate with it the sequence    $\pi|_{[1]}, \pi|_{[2]}, \ldots \pi|_{[n-1]}, \pi|_{[n]}$, where $\pi|_{[i]}$ is the permutation $\pi$ restricted to the elements in $[i]$. 
Note that if $\pi$ is in $\mathcal{S}$, all $\pi|_{[i]}$ are as well.
Thus, a permutation $\pi$ in  $\mathcal{S}_n$ is created by inserting the element $n$ in some \emph{allowed position} in the permutation $\pi|_{[n-1]}$.
In the following, we will use this observation in order to identify every permutation $\pi|_{[i]}$, $i \in [n]$, with an element of the alphabet $\left\lbrace 2, 3, \ldots\right\rbrace \cup \left\lbrace B, E \right\rbrace$ that encodes where the largest element of $\pi|_{[i]}$ was inserted in $\pi|_{[i-1]}$.
This will be done in such a way that the resulting word bijectively encodes the permutation $\pi$.

For this purpose, let us consider a permutation $\pi$ of length $n-1$ that lies in $\mathcal{S}$.
Where are we allowed to insert the element $n$ without creating one of the forbidden patterns?
\begin{itemize}
\item If $\pi$ contains a $132$-pattern on the elements $abc$, then the element $n$ may not be inserted to the left of $a$ since this would create a $4132$-pattern and not between $a$ and $b$ since this would create a $1432$-pattern. 
In other words, the element $n$ must be inserted somewhere to the right of $b$.
\item Similarly, if $\pi$ contains a $231$-pattern on the elements $abc$, the element $n$ must be inserted somewhere to the right of $b$. Otherwise a $4231$- or a $2431$-pattern would be created.
\end{itemize}
Thus the position of the rightmost  element $b$ that plays the role of the 3 in a $132$- or a $231$-pattern in $\pi$ tells us where the element $n$ may be inserted without creating any of the forbidden patterns.
Indeed, if the element $b$ is at the $i$-th position in $\pi$, the element $n$ can be inserted at any of the positions $i+1$ up to $n$.

\begin{example}
In the permutation $\pi= 2413$ the elements $241$ form a 231-pattern and the elements $243$ form a 132-pattern. 
In both patterns the element 4 plays the role of the 3 and thus all positions to the right of 4 are allowed.
Indeed, we can insert 5 at the third, fourth and fifth position in $\pi$ and obtain the following permutations in $\mathcal{S}_5$: 24513, 24153 and 24135.
\demo \end{example}

Since this position will be the key to describing a permutation in $\mathcal{S}$ by a code word in a unique way, let us set up the following notation:
\[
p(\pi) \colonequals 
\max \left\lbrace i \in [n]: \exists j < i < k \text{ s.t. } \pi(j) < \pi(i) \text{ and } \pi(k) < \pi(i) \right\rbrace,
\] 
where we set $\max(\emptyset) \colonequals 0$.


Now, given a permutation $\pi$ and its associated sequence of permutations $\pi|_{[2]}, \ldots, \pi|_{[n]}$, we identify it with a word $w=w_1  \ldots w_{n-1}$ on the alphabet $\left\lbrace 2, 3, \ldots\right\rbrace \cup \left\lbrace B, E \right\rbrace$ in the following way. For all $i \in [n-1]$ we define:
\[
w_i 
\colonequals 
\begin{cases}
p ( \pi|_{[i+1]} )  & \text{if } p(\pi|_{[i+1]}) \neq 0 \\
B & \text{if } p(\pi|_{[i+1]}) = 0 \text{ and } \\
&(i+1) \text{ lies at the beginning of } \pi|_{[i+1]} \\
E & \text{if } p(\pi|_{[i+1]}) = 0 \text{ and } \\
&(i+1) \text{ lies at the end of } \pi|_{[i+1]} \\
\end{cases}
\]

\begin{example}
Consider the permutation $\pi=245178396$ in $\mathcal{S}_9$. As shown in the table on the left hand side of Figure~\ref{CentralBinomial/fig:example_path} on page~\pageref{CentralBinomial/fig:example_path}, it can be identified with the word $w=BE233568$ in $W_8$.
\demo \end{example}

Let us note that two distinct permutations can never be identified with the same code word.
Indeed, if the permutation $\pi|_{[i]}$ is given and if we know that $p(\pi|_{[i+1]})$ is equal to some element in $\left\lbrace 2, 3, \ldots\right\rbrace \cup \left\lbrace B, E \right\rbrace$, there is only a single possibility for placing the element $(i+1)$ in $\pi|_{[i]}$ in order to obtain $\pi|_{[i+1]}$.
Thus the map that sends an element in $\mathcal{S}_n$ to a word of length $n-1$ is injective.

Now we need to show that this map actually associates permutations in $\mathcal{S}$ with \textit{code words} as described in Definition~\ref{CentralBinomial/def:codewords}.
The first condition of code words is clearly fulfilled: First, $w_1=B$ for $\pi|_{[2]}=21$ and $w_1=E$ for $\pi|_{[2]}=12$.
Second, if $w_i \neq B,E$ we have $2 \leq w_i \leq i$ since $2 \leq p(\pi|_{[i+1]}) \leq i$.
In order to show that the second and third condition are also fulfilled we need to identify at which positions new $132$- or $231$-patterns can be created when the element $i$ is inserted in a permutation in $C_{i-1}$.
\begin{itemize}
\item If the element $i$ is inserted at some position $j$ with $1<j<i$ then it automatically creates a 132- or 231-pattern in which the element $i$ plays the role of the 3. 
Thus, irrespective of the value of $w_{i-1}$, we have $p(\pi|_{[i]})=j$.
Since $i$ can only be inserted into $\pi|_{[i-1]}$ to the right of the $p(\pi|_{[i-1]}$-th position, we always have that $p(\pi|_{[i]})=w_i > p(\pi|_{[i-1]})$.
\item If the element $i$ is inserted at the beginning or at the end of the permutation, no new 132- or 231-patterns are created. 
Thus if there were no occurrences of such patterns before, we have that $w_{i+1}$ is equal to $B$ or $E$, depending on whether $i$ is inserted at the beginning or at the end.

However, if $p(\pi|_{[i-1]}) \neq 0$, that is if $w_{i-1}$ is neither $B$ nor $E$, we have $p(\pi|_{[i]}) = p(\pi|_{[i-1]})+1$ if the element $i$ is inserted at the beginning and $p(\pi|_{[i]}) = p(\pi|_{[i-1]})$ if it is inserted at the end.
Together with the fact that $w_i \neq B,E$ whenever $i$ is not inserted at the first or last position, this implies that the second condition of code words is fulfilled.
Furthermore, we also have in this case that $w_i \geq w_{i-1}$ whenever $w_{i-1}$ is an integer and thus the third condition of code words is also fulfilled.
\end{itemize}

Finally, it remains to show that this map is surjective.
Given a code word $w$ of length $n-1$ it is straightforward how to construct the corresponding permutation $\pi$ in a recursive manner.
Start with the element 1 and then place the elements 2 up to $n$ one after the other at the positions prescribed by $w$: if $w_i=B$ the element $i+1$ is placed at the beginning of the permutation, if $w_i=E$ it is placed at the end of the permutation and if $w_i=j$ for some integer $2 \leq j \leq i$ it is placed at the $j$-th position of the permutation. 
That the resulting permutation $\pi$ is an element of $\mathcal{S}$ and avoids all forbidden patterns follows from the discussion at the beginning of this proof.
\end{proof}

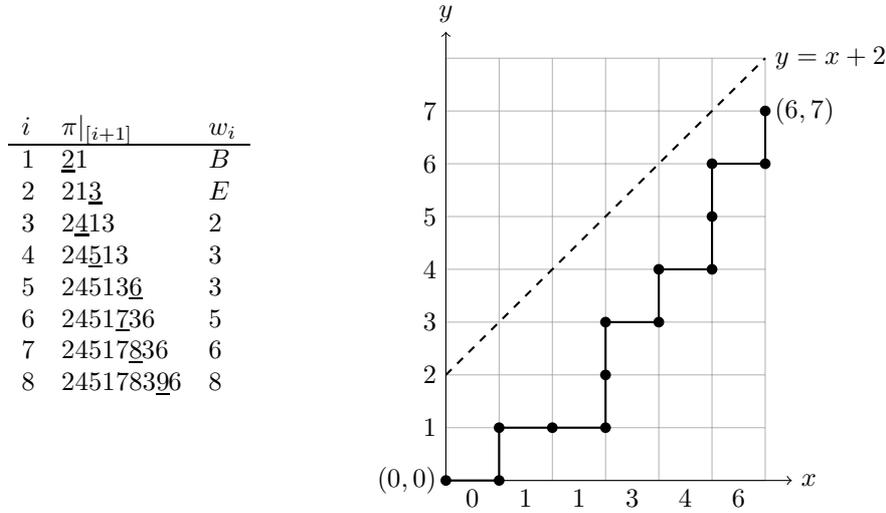
\begin{figure}
\begin{minipage}{0.3\textwidth} 
\centering
$\begin{array}{lll}
i & \pi|_{[i+1]} & w_i \\ \hline 
1 & \underline{2}1 & B \\ 
2 & 21\underline{3} & E\\
3 & 2\underline{4}13 & 2 \\
4 & 24\underline{5}13 & 3 \\
5 & 24513\underline{6} &3 \\
6 & 2451\underline{7}36 & 5\\
7 & 24517\underline{8}36 & 6\\
8 & 2451783\underline{9}6 & 8\\
\end{array}$
\end{minipage}
\hfill
\begin{minipage}{0.6\textwidth}
\centering
\begin{tikzpicture}[scale=0.7]

  	\tikzstyle{axes}=[]
	\tikzstyle{1}=[circle, fill=black, inner sep=0.05cm]

 \draw[style=help lines,draw=black!30,step=1cm] (0,0) grid (6,8);
 
\begin{scope}[style=axes]
    \draw[->] (0,0) -- (6.5,0) node[right] {$x$};
    \draw[->] (0,0) -- (0,8.5) node[above] {$y$};
\end{scope}

\draw[thick, dashed] (0,2) -- (6,8) node[right] {$\textcolor{black}{ y=x+2}$};

\draw[thick]    (0,0) node[1] {}
 -- (1,0) node[1] {}
 -- (1,1) node[1] {}
 -- (2,1) node[1] {}
 -- (3,1) node[1] {}
 -- (3,2) node[1] {}
 -- (3,3) node[1] {}
 -- (4,3) node[1] {}
 -- (4,4) node[1] {}
-- (5,4) node[1] {}
-- (5,5) node[1] {}
-- (5,6) node[1] {}
-- (6,6) node[1] {}
-- (6,7) node[1] {};

\node[right] at (6,7) {$(6,7)$};
\node[left] at (0,0) {$(0,0)$};

\foreach \x/\y in { 0/0, 1/1, 2/1, 3/3, 4/4, 5/6}
    \node[below] at (\x +0.5,0) {$\y$};

\foreach \x in { 1,2,3,4,5,6,7}
    \node[left] at (0,\x ) {$\x$};
    
\end{tikzpicture}
\end{minipage}
\caption{The permutation $\pi=245178396$ in $\mathcal{S}_9$ can be identified with the word $w=BE233568$ in $W_8$ as shown in the table on the left hand side.
The sequence of integers that remains when the initial segment $BE$ has been removed from $w$  and when all integers are decreased by $2$ is $011346$.
It can be represented as the lattice path shown on the right-hand side that starts at $(0,0)$ and ends at $(6,7)$ without ever touching the dashed line.}
\label{CentralBinomial/fig:example_path}
\end{figure}

\begin{lemma}
For all $n \in \mathbb{N}$ the following statement holds:
\[
W_n = \binom{2n}{n}.
\]
\end{lemma}

\begin{proof}
First of all, let us note that 
\[
W_n = \sum_{i=1}^{n}2^i \cdot w(n-i,i)
\]
where $i$ corresponds to the length of the initial segment consisting of letters $B$ and $E$ and where $w(n-i,i)$ is the number of words $w=w_1 \ldots w_{n-i}$ of length $(n-i)$ where $2 \leq w_j \leq j+i$ and $w_j \leq w_{j+1}$.
These numbers $w(n-i,i)$ can be easily computed by viewing the words that they count as certain lattice paths.

First, let us remark that we can decrease every integer in such a word by $2$ and counts the number of words  of length $(n-i)$ where $0 \leq w_j \leq j+i-2$ and $w_j \leq w_{j+1}$.
Now we can represent each such word by a lattice path starting at $(0,0)$, ending at $(n-i,n-1)$ and consisting of unit steps to the East or the North as follows:
If $w_j=k$ for some integer $k$, the $j$-th step to the East is at height $k$, i.e., we take a step from $(j-1,k)$ to $(j,k)$.
Since it holds that $w_j \leq w_{j+1}$, the remaining gaps can simply be filled in with North-steps and there is no need for South-steps.
The condition that  $w_j \leq j+i-2$ is easily translated into the condition that the lattice path must always stay below and never touch the line $y=x+i$.

It is now easy to convince oneself that this correspondence between words of length $(n-i)$ where $2 \leq w_j \leq j+i$ and $w_j \leq w_{j+1}$ and North-East-lattice paths from $(0,0)$ to $(n-i,n-1)$ that never touch the line $y=x+i$ is one-to-one.

\begin{example}
For an example of this correspondence, see the path corresponding to $011346$ that is represented on the right-hand side of Figure~\ref{CentralBinomial/fig:example_path}.
\demo
\end{example}

Such lattice paths have been widely studied in the literature and it is shown, for instance in the first chapter of Mohanty's monograph on lattice path counting~\cite{mohanty1979lattice}, that their number is:
\begin{align*}
w(n-i,i)	& = \binom{2n-i-1}{n-1} \cdot \frac{i}{n}
\end{align*}
Since it is so simple and beautiful, let us briefly repeat the argument of the bijective proof here.
It is known as Andr\'{e}s \textit{reflection principle}.
First, let us remark that 
\begin{align}
w(n-i,i) = \binom{2n-i-1}{n-1} - r(n-i,i),
\label{CentralBinomial/eqn:number_paths}
\end{align}
where the binomial coefficient counts all North-East-lattice paths from $(0,0)$ to $(n-i,n-1)$ and $r(n-i,i)$ is the number of such paths that do touch the line $y=x+i$ at some point.
Given such a ``bad'' path from $(0,0)$ to $(n-i,n-1)$, we denote by $X$ the point where it touches the forbidden line for the first time.
From this path we now construct a new path as follows: The part to the left of $X$ is reflected about the line $y=x+i$, thus turning North-steps into East-steps and vice-versa; the part to the right of $X$ stays the same.
The new path is then a path from $(-i,i)$ to $(n-i,n-1)$ and this transformation is a one-to-one correspondence to all such paths.
Therefore, we have
\begin{align*}
w(n-i,i)	& = \binom{2n-i-1}{n-1} - \binom{2n-i-1}{n}, 
\end{align*}
and thus the enumeration formula stated in Equation~\eqref{CentralBinomial/eqn:number_paths}. Finally, this leads to:
\begin{align*}
W_n & = \sum_{i=1}^{n}2^i \cdot \binom{2n-i-1}{n-i} \cdot \frac{i}{n} \\
	& = \frac{1}{n}\cdot \sum_{j=0}^{n-1}2^{n-j} \cdot \binom{n+j-1}{j} \cdot (n-j) \\
\end{align*}
In order to prove that $W_n=\binom{2n}{n}$ let us first remark that a more general statement holds.
It holds that 
\begin{align}
\sum_{j=0}^{n-1}2^{m-j} \cdot \binom{m+j-1}{j} \cdot (m-j) & = n \cdot  2^{m-n+1} \binom{m+n-1}{n}
\label{CentralBinomial/eqn:central_binomial}
\end{align}
for all $m, n \in \mathbb{N}$ as can be showed very easily by induction over $n$ and every fixed $m$.
The induction start for $n=1$ is trivially true since the left hand side of Equation~\eqref{CentralBinomial/eqn:central_binomial} is equal to $2^m \cdot m$ and the right hand side to $2^m \cdot \binom{m}{1}$.
For the induction step, let us assume that Equation~\eqref{CentralBinomial/eqn:central_binomial} has been proven for $n$.
For $n+1$, we have:%
\begin{align*}
& \sum_{j=0}^{n}2^{m-j} \cdot \binom{m+j-1}{j} \cdot (m-j)= \\= & \, n \cdot  2^{m-n+1} \binom{m+n-1}{n} + 2^{m-n}\binom{m+n-1}{n}\cdot (m-n) \\
 = &  2^{m-n} \binom{m+n-1}{n}\cdot \left( 2n +m-n\right) \\
 = &  2^{m-n} \binom{m+n}{n+1}\cdot \frac{n+1}{m+n}\cdot( m+n), \\
\end{align*}
which proves Equation~\eqref{CentralBinomial/eqn:central_binomial} for $n+1$.

Setting $m=n$ in Equation~\eqref{CentralBinomial/eqn:central_binomial} leads to the central binomial coefficients and finishes the proof.
\end{proof}

\paragraph*{Future work}

The proof presented above first constructs a bijection to a certain set of words and then counts these words by establishing a correspondence of certain segments of these words to lattice paths.
However, it remains open to construct a direct bijection to some combinatorial object that is also counted by the central binomial coefficients. 
For instance, it would be nice to have a direct bijection to granny walks of length $n$.
These  are direct routes from the point $(0,0)$ to the point $(n,n)$ in an integer grid.
See the entry of the sequence A000984 in the On-line Encyclopedia of Integer Sequences~\cite{oeis}, for a (non-exhaustive) list of combinatorial objects counted by the central binomial coefficients.

The permutation class studied in this paper is not the only one known to be enumerated by the central binomial coefficients.
Indeed, in \cite{albert2005insertion} the authors describe a general scheme to enumerate permutation classes, the \emph{insertion encoding}.
As an example of a permutation class amenable to this technique, they present the class defined by the forbidden patterns $3124, 4123, 3142 \text{ and } 4132$.
This class is also enumerated by the central binomial coefficients.

In total, there $\binom{24}{4}=10 \, 626$ different permutation classes defined by a set of four forbidden patterns of length four.
However, this number can be reduced significantly by considering symmetry operations on permutations.
Indeed, if $\sigma^r$ denotes the reverse of the permutation $\sigma$, i.e., the permutation $\sigma$ read from right to left, it is clear that a permutation $\tau$ contains the pattern $\pi$ iff the permutation $\tau^r$ contains the pattern $\pi^r$.
The same holds for the complement permutation $\sigma^c$, i.e., the permutation defined by mapping the element $i$ to $n-\sigma(i)+1$, and the inverse permutation $\sigma^{-1}$, i.e., the permutation mapping the element $\sigma(i)$ to $i$.
Considering these operations and compositions of these operations, we obtain $1\, 524$ so-called \emph{trivial Wilf-equivalence classes}.
Preliminary calculations show that, among these $1\, 524$ permutation classes, there are twelve classes that are possibly enumerated by the central binomial coefficients.
Besides the class treated in this paper and the one mentioned by Albert et. al.\ in \cite{albert2005insertion} these classes are defined by the following forbidden patterns:
\begin{multicols}{2}
\begin{itemize}
\item $1234, 1243, 1324, 1342$
\item $1234, 1243, 1342, 1423$
\item $1234, 1243, 1342, 2341$
\item $1243, 1324, 1342, 1423$
\item $1243, 1324, 1342, 1432$
\item $1243, 2143, 2413, 2431$
\item $1324, 1342, 1423, 1432$
\item $1324, 1342, 1432, 4132$
\item $1342, 1423, 1432, 2431$
\item $1342, 2413, 2431, 3142$
\end{itemize}
\end{multicols}

One of these classes can be very easily shown to be enumerated by the central binomial coefficients, namely the one defined by the four base patterns $1324, 1342, 1432$ and $4132$.
A permutation of length $n$ avoids all four of these patterns if it does not contain the pattern $132$ or if it does contain this pattern it is the element $n$ that plays the role of the $3$ in $132$.
Thus, such a permutation can be obtained by inserting the element $n$ in any of the $n$ possible positions in a $132$-avoiding permutation of length $(n-1)$.
Since $132$-avoiders are counted by the Catalan numbers we have the following:
\[
S_n(1324, 1342, 1432, 4132)= n \cdot \frac{1}{n} \binom{2(n-1)}{n-1}=b_{n-1}.
\]

It would be desirable to prove that all these classes are indeed enumerated by the central binomial coefficients. 
Even more so, it would be desirable to develop a general framework that can be used to prove that all these classes are indeed counted by the central binomial coefficients.

\bibliographystyle{abbrv}
\bibliography{Database}

\end{document}